\documentclass[10pt]{amsart}
\usepackage{amssymb}
\usepackage{epsfig}
\usepackage{url}
\usepackage{setspace}
\usepackage{color}
\usepackage{epstopdf}
\usepackage{graphicx}
\theoremstyle{plain}

\newtheorem{thm}{Theorem}[section]

\newtheorem{prop}[thm]{Proposition}
\newtheorem{rem}[thm]{Remark}
\newtheorem{ques}[thm]{Question}
\newtheorem{conj}[thm]{Conjecture}

\def\cal{\mathcal}
\def\bbb{\mathbb}
\def\op{\operatorname}
\renewcommand{\phi}{\varphi}
\newcommand{\R}{\bbb{R}}
\newcommand{\N}{\bbb{N}}
\newcommand{\Z}{\bbb{Z}}
\newcommand{\Q}{\bbb{Q}}

\newcommand{\C}{\bbb{C}}

\begin{document}

\title[Points at rational distances]{Points at rational distances from the vertices of certain geometric objects}
\author{Andrew Bremner, Maciej Ulas}

\thanks{The research of the second author is supported by the grant of the Polish National Science Centre no. UMO-2012/07/E/ST1/00185}

\begin{abstract}
We consider various problems related to finding points in $\Q^{2}$ and in $\Q^{3}$
which lie at rational distance from the vertices of some specified geometric object, for
example, a square or rectangle in $\Q^{2}$, and a cube or tetrahedron in $\Q^{3}$.
\end{abstract}

\maketitle

\section{Introduction}\label{sec0}
Berry~\cite{Be} showed that the set of rational points in the plane with rational distances to
three given vertices of the unit square is infinite. More precisely, he showed that the set of rational
parametric solutions of the corresponding system of equations is infinite; this generalizes some earlier
work of Leech. In a related work, he was able to show that for any given triangle $ABC$ in which the length
of at least one side is rational and the squares of the lengths of all sides are rational, then the set of
points $P$ with rational distances $|PA|$, $|PB|$, $|PC|$ to the vertices of the triangle is dense in the
plane of the triangle; see Berry~\cite{Be1}. However, it is a notorious and unsolved problem to determine
whether there exists a rational point in the plane at rational distance from the {\it four} corners of the unit
square (see Problem D19 in Guy's book~\cite{Guy}). Because of the difficulty of this problem one can ask a slightly
different question, as to whether there exist rational points in the plane which lie at rational distance
from the four vertices of the {\it rectangle} with vertices $(0,0)$, $(0,1)$, $(a,0)$, and $(a,1)$, for $a \in \Q$. This problem is briefly alluded to in section D19 on p. 284 of Guy's book. In section
\ref{sec2} we reduce this problem to the investigation of the existence of rational points on members of
a certain family of algebraic curves $\cal{C}_{a,t}$ (depending on rational parameters $a$, $t$).
We show that the set of $a \in \Q$ for which the set of rational points on $\cal{C}_{a,t}$ is infinite
is dense in $\R$ (in the Euclidean topology).
%
%

Richard Guy has pointed out that there are immediate solutions to the four-distance unit square problem
if the point is allowed to lie in three space $\Q^3$. Indeed, $(\frac{1}{2}, \frac{1}{2}, \frac{1}{4})$
lies at rational distance to the four vertices $(0,0,0)$, $(0,1,0)$, $(1,0,0)$, $(1,1,0)$ of the square.
This observation leads us to consider the more general problem, of points in $\Q^3$ which lie at rational
distance from the four vertices $(0,0,0)$, $(0,1,0)$, $(1,0,0)$, $(1,1,0)$ of the unit square.  In section
\ref{sec3} we show that such points are dense on the line $x=\frac{1}{2}$, $y=\frac{1}{2}$,
and dense on the plane $x=\frac{1}{2}$. Further, there are infinitely many parameterizations
of such points on the plane $x=y$.
In section \ref{sec4}
we consider the general problem of finding points $(x,y,z) \in \Q^3$ with rational distances to the
vertices of a unit square lying in the plane $z=0$ without any assumptions on $x,y,z$. Attempts to show
such points are dense in $\R^3$ have been unsuccessful to date. However, we are able to show that the
variety related to this problem is unirational over $\Q$. In particular, this implies the existence of
a parametric family of rational points with rational distances to the four
vertices $(0,0,0)$, $(0,1,0)$, $(1,0,0)$, $(1,1,0)$ of the unit square. Whether there exist points in $\Q^3$
at rational distance from the eight vertices of the unit {\it cube} is another seemingly intractable
problem which we leave as open and certainly worthy of further investigation.

In section \ref{sec5} we consider the problem of finding points in $\Q^3$ at rational distance from the
vertices of a general tetrahedron (with rational vertices) and prove that the corresponding algebraic
variety is unirational over $\Q$. This is related to section D22 in Guy's book. This result, together with
the construction of a parameterized family of tetrahedra having rational edges, face areas, and volume (an
independent investigation), leads to constructing a double infinity of sets of five points in $\Q^3$
with the ten distances between them all rational.

Finally, in the last section we collect some numerical results and prove that under a certain symmetry
assumption it is possible to find a parametric family of points in $\Q^3$ with rational distances
to the six vertices of the unit cube. Without symmetry, we found just one point with five of the distances
rational.

\section{Points in $\Q^2$ with rational distances from the vertices of rectangles}\label{sec2}

Let $a \in \Q$. Consider the rectangle $\cal{R}_{a}$ in the plane with vertices at $P_1=(0,0)$, $P_2=(0,1)$, $P_3=(a,0)$,
and $P_4=(a,1)$.

\begin{thm}
\label{thm2-1}
The set of $a\in\Q$ such that there are infinitely many rational points with rational distance to each
of the corners $P_1,...,P_4$ of $\cal{R}_{a}$ is dense in $\R$.
\end{thm}
\begin{proof}
Let $M=(x,y)$ be a rational point with rational distance to each vertex $P_1,...,P_4$ of $\cal{R}$.
This determines the following system of equations:
\begin{equation}\label{rectanglesys1}
\begin{cases}
\begin{array}{lll}
  x^2+y^2 & = & P^2=|MP_{1}|^2,  \\
  x^2+(1-y)^2 & = & Q^2=|MP_{2}|^2, \\
  (a-x)^2+y^2 & = & R^2=|MP_{3}|^2, \\
  (a-x)^2+(1-y)^2 & = & S^2=|MP_{4}|^2.
\end{array}
\end{cases}
\end{equation}
From the first and third equations, and the first and second equations, we deduce respectively
\begin{equation}\label{xy}
x=\frac{1}{2a}(a^2+P^2-R^2),\quad y=\frac{1}{2}(P^2-Q^2+1).
\end{equation}
Eliminating $x,y$ from the system (\ref{rectanglesys1}) we obtain
\begin{equation}\label{rectanglesys2}
\begin{cases}
\begin{array}{lll}
 P^2-Q^2=R^2-S^2,  \\
  a^2 (R^4+a^2+1) + Q^4 + (1+a^2) S^4 = 2 Q^2 (a^2+S^2) + 2 a^2 R^2 (S^2+1).
\end{array}
\end{cases}
\end{equation}
The first quadric may be parameterized by
\begin{equation}
\label{RS}
R=\frac{(P+Q)t^2+P-Q}{2t},\quad S=\frac{(P+Q)t^2-P+Q}{2t}.
\end{equation}
On homogenizing, by setting $P=X/Z$, $Q=Y/Z$, the second equation at (\ref{rectanglesys2}) becomes:
\begin{align*}
&(1 - 4 t^2 + 6 t^4 + 16 a^2 t^4 - 4 t^6 + t^8) (X^4 + Y^4) +
 4(t^2 - 1)^3(t^2 + 1) (X^2 + Y^2) X Y-\\
& 8 a^2 t^2 (1 + t^2)^2(X^2 + Y^2) Z^2 +
 16 a^2 t^2 Z^2 ((1 - t^4) X Y + (1 + a^2) t^2 Z^2) +\\
&\quad 2 (3 - 4 t^2 + 2 t^4 - 16 a^2 t^4 - 4 t^6 + 3 t^8) X^2 Y^2=0.
\end{align*}
This equation defines a curve $\cal{C}_{a,t}$ of genus three over the field $\Q(a,t)$. It is well known that a curve of genus
at least $2$ defined over a function field has only finitely many points with coordinates in this field. Thus, in order
to prove the theorem we must find some specialization $a_{0}$, $t_{0}$ of the rational parameters $a$, $t$, such that
the corresponding curve $\cal{C}_{a_{0},t_{0}}$, has genus at most $1$. In particular, the curve $\cal{C}_{a,t}$ needs
to have singular points. Denote the defining polynomial of $\cal{C}_{a,t}$ by $F=F(X,Y,Z)$. Now
$\cal{C}_{a,t}$ has singular points when the system of equations
\begin{equation}\label{singsol}
F(X,Y,Z)=\partial_{X}F(X,Y,Z)=\partial_{Y}F(X,Y,Z)=\partial_{Z}F(X,Y,Z)=0
\end{equation}
has rational solutions. In order to find solutions of this system, consider the ideal
\begin{equation*}
\op{Sing}=<F,\partial_{X}F,\partial_{Y}F,\partial_{Z}F>
\end{equation*}
and compute its Gr\"{o}bner basis.
The basis contains the polynomial $-a^2(1+a^2)t^6(1 + 2 a t - t^2) (-1 + 2 a t + t^2) Z^7$, and to obtain
something non-trivial, we require $a=\pm (1-t^2)/2t$.  We choose without loss of generality $a=(1-t^2)/2t$ (the other sign
corresponds to solutions in which $x$ is replaced by $-x$).
Now, $F=G^2$, where
\begin{equation*}
G(X,Y,Z)=(t^2-1)((t^2+1)X^2+2(t^2-1)X Y+(t^2+1)Y^2 -(t^2+1)Z^2)
\end{equation*}
and by abuse of notation we are working with the curve $\cal{C}_{a,t}:\;G(X,Y,Z)=0$ of degree 2 defined over
the rational function field $\Q(t)$. The genus of $\cal{C}_{a,t}$ is 0, and moreover, there is a $\Q(t)$-rational
point $(0,1,1)$ lying on $\cal{C}_{a,t}$. This point allows the parametrization of $\cal{C}_{a,t}$ in the following form:
\begin{equation*}
X=2u((1-t^2)u+(t^2+1)v), \; Y=(t^2+1)(u^2-v^2), \; Z=(t^2+1)(u^2+v^2)-2(t^2-1)u v.
\end{equation*}
Recalling that $P=X/Z, Q=Y/Z$ and using the expressions for $R,S$ at (\ref{RS}), $x,y$ at (\ref{xy}), we get that
for $a=(1-t^2)/2t$ there is the following parametric solution of the system (\ref{rectanglesys1}):
\begin{align*}
x=&\frac{4 t u(v^2-u^2)\left((t^2-1)u-(t^2+1)v\right)}{\left((t^2+1)u^2-2(t^2-1)u v+(1+t^2)v^2\right)^2},\\
y=&\frac{2 u\left((t-1)u-(t+1)v\right)\left((t+1)u-(t-1)v\right)\left((t^2-1)u - (t^2+1)v\right)}{\left((t^2+1)u^2-2(t^2-1)u v+(1+t^2)v^2\right)^2}.
\end{align*}
To finish the proof, note that the rational map $a:\R\ni t\mapsto \frac{1-t^2}{2t}\in\R$ is continuous
and has the obvious property
 \begin{equation*}
\lim_{t\rightarrow -\infty}a(t)=+\infty,\quad\quad \lim_{t\rightarrow +\infty}a(t)=-\infty.
\end{equation*}
The density of $\Q$ in $\R$ together with the properties of $a(t)$ immediately imply that the set $a(\Q)\cap \R_{+}$ is dense in $\R_{+}$ in the Euclidean topology. The theorem follows.
\end{proof}

\begin{rem}
{\rm Observe that the construction presented in the proof of Theorem \ref{thm2-1} allows deduction of the following simple result..

\begin{thm}\label{sqrt{2}}
Let $K$ be a number field and suppose that $\sqrt{2}\in K$. Then the set of $K$-rational points with $K$-rational distances to the vertices of the square $\cal{R}_{1}$ is infinite.
\end{thm}
\begin{proof}
Let $a=1$ and take $t=1+\sqrt{2}$. Then $1+2at-t^2=0$ and using the parametrization constructed at the end of Theorem \ref{thm2-1} (with $v=1$) we get that for
\begin{align*}
x=&\frac{u(u-\sqrt{2})(1-u^2)}{(u^2-\sqrt{2}u+1)^2},\\
y=&\frac{(3-2\sqrt{2})u(\sqrt{2}(u-1)-2)(\sqrt{2}(u-1)+\sqrt{2})((1+\sqrt{2}) u-\sqrt{2}-2)}{2(u^2-\sqrt{2}u+1)^2}
\end{align*}
and any given $u\in K$ such that $\sqrt{2}u^2-2u+\sqrt{2}\neq 0$, the distance of the point $P=(x,y)$ to the vertices $P_{1}, P_{2}, P_{3}, P_{4}$ of $\cal{R}_{1}$ is $K$-rational.
\end{proof}

}
\end{rem}

\begin{rem}
{\rm The construction of $a$'s and the corresponding solutions $x,y$ of the system (\ref{rectanglesys1}) presented in the proof of Theorem \ref{thm2-1} has one aesthetic disadvantage. In order that $(x,y)$ lie {\it inside} the rectangle $\cal{R}$, it is necessary that
$x$, $a-x$, $y$, $1-y$, all be positive. However,
\begin{align*}
x & (a-x)y(1-y) = -4 u^2(u^2-v^2)^2 \left( ((1-t)u+(1+t)v)((1+t)u+(1-t)v) \right)^2 \times \\
& \left( \frac{((1-t^2)u+(1+t^2)v)((1+2t-t^2)u+(1+t^2)v)((1-2t-t^2)u+(1+t^2)v)}{((1+t^2)u^2+2(1-t^2)u v+(1+t^2)v^2)^4} \right)^2
\end{align*}
which is evidently negative. Thus the point $(x,y)$ can never lie within the rectangle $\cal{R}$.
A natural question arises therefore as to whether it is possible to find a positive
rational number $a$ such that the system (\ref{rectanglesys1}) has rational solutions $x, y$
with $x$, $a-x$, $y$, $1-y$, all positive?
The answer is yes, on account of the family
$$
a=\frac{2t}{t^2-1},\quad x=\frac{t}{t^2-1},\quad y=\frac{1}{2}
$$
where $x, a-x, y, 1-y$ are all positive when $t>1$; however, this family is rather uninteresting,
in that correspondingly $P=Q=R=S$. An equivalent question was posed by Dodge in \cite{Dod} with an answer
given by Shute and Yocom. They proved that if $p_{i}, q_{i}, r_{i}$ are Pythagorean triples for
$i=1,2$, and $A=p_{1}q_{2}+p_{2}q_{1}, B=p_{1}p_{2}+q_{1}q_{2}$, then the point
$M=(p_{1}q_{2}, q_{1}q_{2})$ lies inside the rectangle with vertices $(0,0)$, $(A,0)$, $(0,B)$,
$(A,B)$, and, moreover, the distances of $M$ to the vertices of the rectangle are rational.
Using their result one can prove that the set of those $a\in\Q$, such that there are infinitely
many rational points inside the rectangle $\cal{R}_{a}$ with rational distance to its vertices,
is dense in $\R_{+}$. Indeed, note that the point
$$
P=\left(\frac{p_{1}q_{2}}{B},\frac{q_{1}q_{2}}{B}\right)
$$
lies inside the rectangle $\cal{R}_{a}$, with $a=A/B$. To finish the proof, it is enough to show
that one can find infinitely many Pythagorean triples $p_{i}, q_{i}, r_{i}, i=1,2$, such that
$a=A/B$ is constant. Put
\begin{equation*}
\begin{array}{lll}
  p_{1}=1-U^2, & q_{1}=2U, & r_{1}=1+U^2, \\
  p_{2}=1-V^2, & q_{2}=2V, & r_{2}=1+V^2
\end{array}
\end{equation*}
and then
\begin{equation*}
A(U,V)=2(U + V) (1-UV),\quad B(U,V)=(1+U-(1-U)V) (1 - U+ (1+U)V).
\end{equation*}
Since the rectangles $\cal{R}_{a}$  and $\cal{R}_{1/a}$ are equivalent under rotation by ninety degrees
and scaling, we consider only the case $0<a<1$. Set $a=a(t)=\frac{2t}{1-t^2}$, with $0<t<\sqrt{2}-1$
(the transformation between $\cal{R}_{a}$ and $\cal{R}_{1/a}$ is now given by $t \leftrightarrow \frac{1-t}{1+t}$).
Define $C_t$ to be the curve $A(U,V)=a(t)B(U,V)$:
\[ C_t: \; (U+V)(1-U V)(1-t^2) - t (1+U-(1-U)V)(1-U+(1+U)V) = 0. \]
The triple $(t,U,V)$ corresponds to a point $P$ with rational distances to the vertices
of $\cal{R}_{a}$ (with $a=a(t)$) precisely when
\begin{equation}\label{tUVconditions}
0<\frac{p_{1}q_{2}}{B}<\frac{A}{B},\quad 0<\frac{q_{1}q_{2}}{B}<1,
\end{equation}
that is, when
\begin{equation}
\label{tUVconditions1}
\frac{V(1-U^2)}{\Delta}>0, \quad \frac{U(1-V^2)}{\Delta}>0, \quad \frac{U V}{\Delta}>0, \quad \frac{(1-U^2)(1-V^2)}{\Delta}>0,
\end{equation}
where
\[ \Delta=(1+U-(1-U)V)(1-U+(1+U)V)=(1-U^2)(1-V^2)+4U V. \]
Our strategy is to show that the curve $C_t$ contains infinitely many rational points in the
unit square $0<U<1$, $0<V<1$, when the inequalities (\ref{tUVconditions1}) clearly hold, so that the
inequalities (\ref{tUVconditions}) will follow. \\

The equation for $C_t$ defines the hyperelliptic quartic curve:
$$
\cal{C}_{t}: \; W^2=((t^2-1)U^2-4t U+ 1-t^2)^2+4(t U^2-(t^2-1)U-t)^2,
$$
where $W=2(t-U)(1+t U)V + ((t-1)U-t-1)((t+1)U+t-1)$. Now $\cal{C}_{t}$ contains the point $R=(0,\;t^2+1)$,
and a cubic model $\cal{E}_{t}$ for $\cal{C}_{t}$ is given by
\begin{equation*}
\cal{E}_{t}: \; Y^2=X(X+(t^2-2 t-1)^2)(X+(t^2+2 t-1)^2).
\end{equation*}
The curve $\cal{E}_{t}$ contains the point $H(X,Y)=(-(1+t^2)^2, \; 4t(1-t^4))$, and it is readily checked
that $H$ is of infinite order in $\cal{E}_{t}(\Q(t))$.
We now apply theorems of Silverman~\cite[p. 368]{Sil} and of Hurwitz~\cite{Hur}
(see also Skolem~\cite[p. 78]{Sko}).
Silverman's theorem states that if $E_{t}$ is an elliptic curve
defined over $\Q(t)$ with positive rank, then for all but finitely
many $t_{0}\in\Q$, the curve $E_{t_{0}}$ obtained from the
curve $E_{t}$ by the specialization $t=t_{0}$ has positive rank.
From this result it follows that for all but finitely many $t_0\in\Q$ the
elliptic curve $\cal{E}_{t_0}$ is of positive rank. (Indeed, a straightforward
computation shows that the specialization of $H$ at $t=t_0$ is of infinite order
in $\cal{E}_{t_0}(\Q)$ for all $t_0 \in \Q$ with $t_0 \neq 0, \pm 1$, that is, for
all $t_0$ giving a nonsingular specialization).

The theorem of Hurwitz states that if an elliptic curve $E$ defined
over $\Q$ has positive rank and one torsion point of order two (defined over the field $\R$) then the
set $E(\Q)$ is dense in $E(\R)$. The same result holds if $E$ has
three torsion points (defined over the field $\R$) of order two under the assumption that there is
a rational point of infinite order on the bounded branch of the set $E(\R)$.  Here, for $0<t<1$,
the point $H$ satisfies this latter condition, since for $0<t<1$ we have
\[ -(-1-2t+t^2)^2 < -(1+t^2)^2 < -(-1+2t+t^2)^2. \]

Applying the Hurwitz theorem
we get that for all but finitely many $t_0\in\Q$ the set $\cal{E}_{t_0}(\Q)$ is dense
in the set $\cal{E}_{t_0}(\R)$. This proves that the set $\cal{E}_{t_0}(\Q)$ is dense in
the set $\cal{E}_{t_0}(\R)$ in the Euclidean topology. As a consequence we get that the set
$\cal{C}_{t_0}(\Q)$ is dense in the set $\cal{C}_{t_0}(\R)$.
This immediately implies that the image of the map
\begin{equation*}
\cal{C}_{t_0}(\Q)\ni (U,W)\mapsto U\in\R
\end{equation*}
is dense in $\R$ for all but finitely many $t_0\in\Q$ (which is a consequence of the positivity
of the polynomial defining the quartic $\cal{C}_{t}$).

In order to finish the proof therefore we need to show that for given rational $t\in (0,\sqrt{2}-1)$
we can find infinitely many rational points $(U,V)\in \cal{C}_{t}(\Q)$ satisfying
$0<U<1$ and $0<V<1$.
Now
\[ V=V(U)=(W(U) -((t-1) U-(t+1))((t+1)U+t-1))/(2(t-U)(1+t U)), \]
and we consider the connected component of the curve that passes through the point $(U,V)=(0,t)$,
certainly a continuous function on the interval $0<U<t$.  Using
\[ \frac{dV}{dU} =  \frac{-1+t^2-2t U+4t V+2U V-2t^2U V+V^2-t^2V^2+2t U V^2}{1-t^2-4t U-U^2+t^2U^2+2t V-2U V+2t^2U V-2t U^2V} \]
we compute that $\frac{dV}{dU}(0,t)= - \frac{(-1-2t+t^2)(-1+2t+t^2)}{(1+t^2)} < 0$.
Taking $\frac{dV}{dU}$ in the form
\[ \frac{dV}{dU} = -\frac{(1+U^2)(t-V)V(1+t V)}{(1+V^2)(t-U)U(1+t U)}, \]
then the derivative can vanish for $0<U<t$ only when $V=-1/t$ (forcing $U=0$), $0$ (with $U=-1/t,t$),
$t$ (with $U=0$).  Accordingly $\frac{dV}{dU}$ has constant sign (negative) for $0<U<t$,
so that $V(U)$ is a decreasing function on the interval $0 \leq U < t$. Accordingly,
$0 \leq U < t$ implies $0 < V \leq t$ on this component of the curve. Thus the curve $C_t$
contains infinitely many rational points in the square $0<U<t$, $0<V<t$. The situation
is graphed in Figure 1.
\begin{figure}
\includegraphics[width=6.0cm, height=6.0cm, angle=-90]{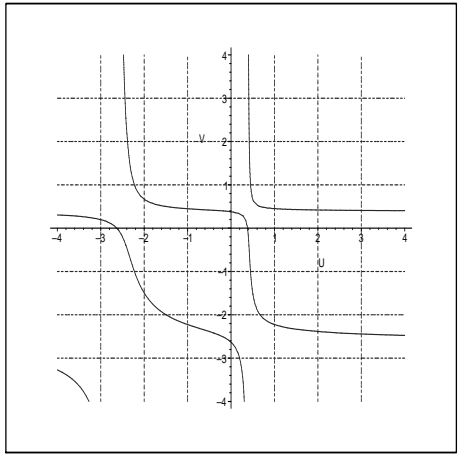}
\caption{The graph of $C_t$ when $0<t<\sqrt{2}-1$}
\end{figure}

Summing up, for all but finitely many $t\in (0,\sqrt{2}-1)$ we can find infinitely many rational
points satisfying the conditions (\ref{tUVconditions1}) and the equation $A(U,V)=a(t)B(U,V)$.
This implies that for all but finitely rational numbers $t\in (0,\sqrt{2}-1)$, the corresponding
point $P$ lies inside the rectangle $\cal{R}_{a(t)}$. Because of the continuity of the function
$a=a(t)$, we get that the set $a(\Q \cap (0,\sqrt{2}-1))$ with $a(t)$ having the required property,
is dense in the set $\R_{+} \cap (0,1)$.

The earlier remark about the equivalence of the rectangles $\cal{R}_{1/a}$ and $\cal{R}_a$
under rotation and scaling now gives the following theorem.

\begin{thm}
\label{thm2-2}
The set of $a\in\Q$ such that there are infinitely many rational points lying inside the
rectangle $\cal{R}_{a}$ with rational distance to each of the corners $P_1,...,P_4$ of $\cal{R}_{a}$
is dense in $\R_{+}$.
\end{thm}

\begin{rem}
{\rm It is quite interesting that all $a$'s we have found above are of the form $(1-t^2)/2t$ or
$2t/(1-t^2)$. A question arises as to whether we can find $a$'s which are not of this form and
such that there is a rational point with rational distances to the vertices of the rectangle
$\cal{R}_{a}$.
A small numerical search for other such triples $(a,x,y) \in \Q^3$ was undertaken. We wrote
$(x,y)=(X/Z,Y/Z)$, $X,Y,Z>0$, and restricted the search to height of $a$ at most 20, and
$X+Y+Z \leq 1000$. The involutions $(a,x,y) \leftrightarrow (1/a,y/a,x/a)$,
$(a,x,y) \leftrightarrow (a,x-a,y)$, and $(a,x,y) \leftrightarrow (a,x,1-y)$ mean that we can
restrict attention to solutions satisfying $a>1$, $x \leq a/2$, $y \leq 1/2$. Of the solutions
found in the range, fourteen have $x=0$;  seventeen have $x=a/2$; and forty-five have $y=1/2$.
These all imply some equalities between $P,Q,R,S$, and we list only those solutions found where
$P,Q,R,S$ are distinct.

\begin{center}
\begin{tabular}{c|cc||c|cc||c|cc}
 $a$ & $x$ & $y$ & $a$ & $x$ & $y$ & $a$ & $x$ & $y$ \\ \hline
13/12    & 88/399  & 55/133 & 12/11 & 24/77 & 32/77 & 19/12 & 35/204 & 7/17 \\
9/8      & 120/169 & 50/169 & 17/15 & 15/14 & 11/56 & 13/6 & 273/500 & 34/125 \\
9/8  & 15/56 & 5/14  &  &  &  &  &  & \\ \hline
\end{tabular}
\begin{center}
Table 1: points in $\Q^2$ at rational distance to vertices of $\mathcal{R}_a$
\end{center}
\end{center}

}
\end{rem}

\bigskip


}
\end{rem}

\section{Special points in $\Q^3$ at rational distance to the vertices of the unit square}\label{sec3}
We normalize coordinates so that the unit square lies in the plane $z=0$, with vertices
$A=\{(0,0,0),(0,1,0),(1,0,0),(1,1,0)\}$.
\begin{prop}\label{prop1}
Let $\lambda$ be the line in $\R^{3}$ given by $\lambda: x=y=\frac{1}{2}$. Then the set
\begin{equation*}
\Lambda=\{P\in \lambda(\Q):\;\mbox{the distance } |PQ| \mbox{ is rational for all } Q\in A\}
\end{equation*}
is dense in $\lambda(\R)$.
\end{prop}
\begin{proof}
It is clear that $P=(\frac{1}{2},\frac{1}{2},z)\in \Lambda$ if and only if
\begin{equation*}
\frac{1}{4}+z^2=T^2,
\end{equation*}
for some rational $T$. This equation represents a conic with rational point
$(z,T)=(0,\frac{1}{2})$, and so is parameterizable, for example, by:
\begin{equation*}
z=\frac{1-u^2}{4u},\quad T=\frac{1+u^2}{4u}.
\end{equation*}
To finish the proof, note that for the rational map $z:\R\ni u\mapsto \frac{1-u^2}{4u}\in\R$ we have $\overline{z(\Q)}=\R$. This implies that $\Lambda = \{(\frac{1}{2},\frac{1}{2},z(u)):\;u\in\Q\setminus\{0\}\}$ is dense in $\lambda(\R)$.
\end{proof}

%

\begin{thm}\label{prop2}
Let $\pi$ be the plane in $\R^{3}$ given by $\pi : x=\frac{1}{2}$. Then the set
\begin{equation*}
\Pi=\{M\in \pi(\Q):\;\mbox{ the distance } |MR| \mbox{ is rational for all } R\in A\}
\end{equation*}
is dense in $\pi(\R)$.
\end{thm}
\begin{proof}
Points $(\frac{1}{2},y,z)$ which lie in $\Pi$ are in one to one correspondence with
rational points on the intersection of the following two quadric surfaces in $\R^{4}$:
\begin{equation}\label{sys1}
\frac{1}{4}+y^2+z^2=P^2,\quad \frac{1}{4}+(1-y)^2+z^2=Q^2.
\end{equation}
Subtracting the second equation from the first gives $y=\frac{1}{2}(P^2-Q^2+1)$. So, on
eliminating $y$, the problem of finding rational solutions of (\ref{sys1}) is equivalent to
finding rational points on the surface $S$ given by the equation
\begin{equation}
\label{surfaceS}
S:\;4z^2=-2+2P^2-P^4+2(P^2+1)Q^2-Q^4=:H(P,Q).
\end{equation}
From a geometric point of view, the (homogenized version) of the surface $S$ represents
a del Pezzo surface of degree two which is just a blowup of seven points lying in general position
in $\mathbb{P}^{2}$. In particular, this implies that the surface is geometrically rational
which means that it is rational over $\C$. Note that this immediately implies the potential density
of rational points on $S$, which means that there is a finite extension $K$ of $\Q$ such that $S(K)$
is dense in the Zariski topology. However, we are interested in the density of rational points in
the Euclidean topology, and it seems that there is no way to use the mentioned property in order to
address this. We thus provide alternative reasoning. \\ \\
First, from (\ref{sys1}) we have the inequalities $|P| \geq 1/2$, $|Q| \geq 1/2$,
and because $H(P,Q)=H(\pm P, \pm Q)$ we may suppose without loss of generality that $P \geq 1/2$,
$Q \geq 1/2$.
We have the point on $\cal{S}$ defined by $(P_0(u,v),Q_0(u,v),z_0(u,v))=$
\[ \left( \frac{u^4+1-4\frac{u(u^2-1)}{u^2+1}v+2 v^2}{4(u^2-1)v}, \quad \frac{u^4+1+4\frac{u(u^2-1)}{u^2+1}v+2v^2}{4(u^2-1)v}, \quad \frac{u^4+1-2v^2}{4(u^2+1)v} \right); \]
and in the domain $\cal{D}:=\{(u,v)\in\R^{2}:\;u>1,v>0\}$, it is straightforward to verify that $P_0(u,v)$ has a single extremum
at the point
\[ (u_0,v_0)=(\alpha+\alpha^2, \; (1+\alpha)(1+\alpha^2)), \qquad \alpha^2=\frac{1+\sqrt{5}}{2}. \]
This point is a local minimum, with minimum value $P_0(u_0,v_0)=\frac{1}{2}$.
Since $P_0(u,v)$ is a continuous function in $\cal{D}$ and
$\lim_{u \rightarrow 1^+} P_0(u,v) = \lim_{v \rightarrow 0^+} P_0(u,v)=\infty$, it follows that the set of values
$\{P_0(u,v): u \in \Q \cap (1,\infty), v \in \Q \cap (0, \infty)\}$ is dense in the real interval
$(\frac{1}{2}, \infty)$.
Next, consider the equation
\[ C: \; 4Z^2=H(P_0(u,v),Q), \]
which we regard as defining a curve $C$ over $\Q(u,v)$.  The curve possesses the point
$(Q,Z)=(Q_0(u,v), z_0(u,v))$, and has cubic model
\begin{align*}
E: & y^2 = x^3-\big((1+u^2)^2(1+u^4)^2+8u(1-u^8)v+4(5+6u^2-14u^4+6u^6+5u^8)v^2+ \\
& 16u(1-u^4)v^3+4(1+u^2)^2v^4 \big) x^2+16(u^4-1)^2v^2 \big( (1+u^2)(1+u^4)+2(1-u)(1+u)^3v+ \\
& 2(1+u^2)v^2 \big) \big( (1+u^2)(1+u^4)-2(1-u)^3(1+u)v+2(1+u^2)v^2 \big) x.
\end{align*}

It is easy to check that if $u',v'\in\Q$, then the curve $E_{u',v'}$ obtained from $E$ by the
specialization $u=u', v=v'$ is singular only when $u=1$ or $v=0$. However, the sets
$\{1\}\times \Q$, $\Q\times\{0\}$ have empty intersection with $\cal{D}$. Thus, for all
$(u',v')\in (\Q\times\Q)\cap \cal{D}=:\cal{D}'$, the specialized curve $E_{u',v'}$ is an elliptic curve.
Furthermore, we note that
for each $u',v' \in \Q$, $E_{u',v'}$ has three points of order 2 defined over $\R$
(this is a simple consequence of the positivity of the discriminant of the polynomial defining
the curve $E$), with $x$-coordinates $0<r_1<r_2$, so that $(0,0)$ lies on the bounded component of the
curve. \\ \\
The image $R_{u,v}$ on $E$  of the point $(-Q_0(u,v), z_0(u,v))$ is of infinite order as element
of the group $E(\Q(u,v))$. For any given $u'\in\Q\cap (1,\infty)$ it is straightforward to compute
the set of rational numbers $v'\in \Q\cap (0,\infty)$ such that the point $R_{u',v'}$ is of finite order
on $E_{u',v'}$; this set is finite in consequence of Mazur's Theorem.
Applying Silverman's Theorem, for given $u'\in\Q\cap (1,\infty)$ then for all but finitely many
$v'\in \Q$ the point $R_{u',v'}$ is of infinite order on the curve $E_{u',v'}$.

Now choose sequences
$(u_{n})_{n\in\N}$, $(v_{n})_{n\in\N}$ of rational numbers such that
\[ u_{n} \in \Q\cap (1,\infty), \lim_{n\rightarrow +\infty}u_{n}=\alpha+\alpha^2, \quad v_{n} \in \Q\cap (0,\infty), \lim_{n\rightarrow +\infty} v_{n}=(1+\alpha)(1+\alpha^2), \]
so that $\lim_{n\rightarrow +\infty}P_{0}(u_{n},v_{n})=1/2$.

With $R_{u_{n},v'}$ of infinite order on $E_{u_{n},v'}$, then either $R_{u_{n},v'}$ or
$R_{u_{n},v'}+(0,0)$ lies on the bounded component of the curve, and we can apply the Hurwitz Theorem
as before to deduce that the set $E_{u_{n},v'}(\Q)$ is dense in
the set $E_{u_{n},v'}(\R)$. This immediately implies that the set $E(\Q)$ is dense in the set $E(\R)$.
Because $E$ is birationally equivalent to $C$, we get that $C(\Q)$ is dense in $C(\R)$.
Because
$$
\bigcup_{n\in\N}\{P_0(u_{n},v'):  v'\in \Q \cap (0, \infty)\}
$$
is dense in $(\frac{1}{2}, \infty)$, it follows that
$\cal{S}(\Q) \cap \{(P,Q,z): P>1/2\}$ is dense in the Euclidean topology
in the set $\cal{S}(\R) \cap \{(P,Q,z): P>1/2\}$. Our theorem follows.
\end{proof}

\begin{rem}
\rm{
The proof could be simplified if we were able to demonstrate finiteness of the set of
$(u',v') \in \cal{D}'$ for which the point $R_{u',v'}$ is of finite order
in $E(\Q)$, for then the theorems of Silverman and Hurwitz could be applied directly
without the necessity of selecting limiting sequences. However, this computation
is difficult.  By Mazur's Theorem the point $R_{u',v'}$ on $E_{u',v'}$ is of finite order provided
$mR_{u',v'}=\cal{O}$ for some $m\in\{2,\ldots, 10,12 \}$. Let
$$
m R_{u,v}=\left(\frac{x_{m}}{d_{m}^2},\frac{y_{m}}{d_{m}^3}\right),
$$
where $x_{m}, y_{m}, d_{m}\in\Q[u,v]$ for $m=2,\ldots,10,12$. We consider the denominator of the $x$- coordinate of the point $mR_{u',v'}$
and define the curve $C_{m}:\;d_{m}(u,v)=0$. The set $C_{m}(\cal{D}')$ of points in $\cal{D}'$
lying on $C_{m}$ parameterize those pairs $(u',v') \in \cal{D}'$ which lead to $R_{u',v'}$ of order
(dividing) $m$.  Consider the map
$$
\Phi:\;C_{m}(\cal{D}') \ni (u,v)\mapsto u\in\Q.
$$
and put
\begin{equation*}
B:=\bigcup _{m=2}^{12}\Phi(C_{m}(\cal{D}')),
\end{equation*}
where for $m=7, 11$ we put $C_{m}(\cal{D}')=\emptyset$. Indeed, the case $m=7$ is impossible due to the existence of the rational point $(0,0)$ of order 2 on $E_{u',v'}$ and the fact that the torsion group of $E_{u',v'}$ cannot be isomorphic to $\Z_{2}\times \Z_{7}\simeq \Z_{14}$.
From the definition of $B$, if $u'\not\in B$ then the point $R_{u',v'}$ is of infinite order on
$E_{u',v'}$ for all rational $v'>0$. In theory at least it is possible to give a precise description of the set $B$.
Indeed, for given $m$ the polynomial $d_{m}$ may be factorized as
$d_{m}=f_{1,m}\cdot\ldots\cdot f_{k_{m},m}$ in $\Q[u,v]$, where $f_{i,m}$ is irreducible in $\Q[u,v]$.
Thus $d_{m}(u,v)=0$ if and only if $f_{i,m}(u,v)=0$ for some $i\in\{1,2,\ldots,k_{m}\}$. The equation
$f_{i,m}(u,v)=0$ defines an irreducible curve, say $C_{i,m}$, and thus
\begin{equation*}
B=\bigcup_{m=2}^{12}\bigcup_{i=1}^{k_{m}}\Phi(C_{i,m}(\cal{D}'))
\end{equation*}
(where we define $C_{i,7}(\cal{D}')=C_{i,11}(\cal{D}')=\emptyset$).
For example, we have $d_{2}(u,v)=(u^2-1)(u^4 - 2v^2 + 1)f_{4,2}(u,v)f_{5,2}(u,v)$. The curve
$C_{3,2}:\;u^4 - 2v^2 + 1=0$ is of genus 1 and the only rational points on $C_{3,2}$ satisfy
$|u|=|v|=1$. The genus of $C_{4,2}$ is 3 and the genus of $C_{5,2}$ is 19; thus by Faltings's Theorem
these curves contain only finitely many rational points. However, we are unable to compute the
corresponding sets. Matters are even worse for $m\geq 3$. It is a highly non-trivial task to compute
the factorization of $d_{m}$ and even when this has been done, it is still necessary to compute
the genus of the corresponding curves.
When $m=3$ we were able to compute that $d_{3}(u,v)=(u^2-1)(u^4 - 2v^2 + 1)f_{4,3}(u,v)$,
where $f_{4,3}$ is of degree 72. A rather long computation was needed in order to check that
the genus of $C_{4,3}$ is $\geq 65$. To get this inequality we reduce the curve $C_{4,3}$ modulo 5
and observe that $f_{4,3}\in\mathbb{F}_{5}[u,v]$ is irreducible and
$\op{deg}_{\Q[u,v]}f_{4,3}=\op{deg}_{\mathbb{F}_{5}[u,v]}f_{4,3}$. We thus get the inequality
$\op{genus}_{\C}(C_{4,3})\geq \op{genus}_{\mathbb{F}_{5}}(C_{4,3})=65$, where the last equality
was obtained via computation in Magma.
When $m=4$ we have $d_{4}(u,v)=(u^2-1)(u^4 - 2v^2 + 1)f_{4,4}(u,v)f_{5,4}(u,v)$, where
$\op{deg}f_{4,4}=36$ and $\op{deg}f_{5,4}=72$. Using similar reasoning as for $m=3$,
the genus of $C_{4,4}$ is $\geq 29$ and the genus of $C_{5,4}$ is $\geq 113$ (in this case we
performed calculations over $\mathbb{F}_{3}$).  When $m=5$ we have
$d_{5}(u,v)=(u^2-1)(u^4 - 2v^2 + 1)f_{4,5}(u,v)$, where $\op{deg}f_{4,5}=216$. We were unable to
finish the genus calculations in this case: Magma was still running after three days. However,
we expect that these computations can be performed and believe that in each case the genus of
the corresponding curve is $\geq 2$ which would imply (via the Faltings theorem) that the set
$B$ is finite.
}
\end{rem}

\begin{rem}{\rm
The combination of the theorems of Hurwitz and Silverman which allows proof of the density results
is a very useful tool and can be used in other situations too; see~\cite{Be1,BrUl,Ul}.
}
\end{rem}

\begin{rem}{\rm
It is clear that the same result as in Proposition \ref{prop2} can be obtained for the plane given by the
equation $y=\frac{1}{2}$.
}
\end{rem}

We are able to prove the following result (which falls short of being a density statement)
concerning the existence of rational points on the plane $x=y$ with rational distances to
elements of $A$.
\begin{prop}\label{prop3}
Let $\pi$ be the plane in $\R^{3}$ given by $\pi : x=y$. Then the set
\begin{equation*}
\Pi=\{P\in \pi(\Q):\;\mbox{ the distance } |PQ| \mbox{ is rational for all } Q\in A\}
\end{equation*}
contains images of infinitely many rational parametric curves.
\end{prop}
\begin{proof}
We now have
\begin{equation}
\label{prop3.4}
\begin{cases}
\begin{array}{lll}
2 x^2 + z^2    & = &P^2, \\
2 x^2-2x+1+z^2 & = &Q^2, \\
2 (x-1)^2 +z^2 & = & S^2.
\end{array}
\end{cases}
\end{equation}
Thus
\[ P^2-2 Q^2+S^2= 0, \qquad x=1/2+(P^2-Q^2)/2, \qquad z^2 = P^2-2 x^2. \]
The former is parametrized by
\[ \tau \;P = m^2+2m-1, \quad \tau \;Q=m^2+1, \quad \tau \;S=m^2-2m-1, \]
giving
\[ x=1/2-2m(1-m^2)/\tau^2, \qquad (\tau^2z)^2 = 1/2 (\tau^2-8 m^2)(-\tau^2+2(1-m^2)^2). \]
Regard the latter as an elliptic quartic over $\Q(m)$. Under the quadratic base change
$m=4k/(2+k^2)$, the curve becomes, with $\tau=t/(2+k^2)^2$, $Z=t^2 z$,
\[ \cal{C}:\;Z^2 = -\frac{1}{2}(t^2-128k^2(2+k^2)^2)(t^2 - 2(4-12k^2+k^4)^2), \]
which has a point at
\begin{align*}
(t,Z) =  \Big(& \frac{4(2+k^2)^2(4-12k^2+k^4)}{(12-4k^2+3k^4)},   \\
               & \frac{4(4-k^4)(4-12k^2+k^4)(16-352k^2-104k^4-88k^6+k^8)}{(12-4k^2+3k^4)^2} \Big).
\end{align*}
A cubic model of the curve is
\[\cal{E}:\; Y^2 = X(X - (4-16k-12k^2-8k^3+k^4)^2) (X - (4+16k-12k^2+8k^3+k^4)^2), \]
with point of infinite order $Q=(X,Y)$, where
\begin{align*}
X&=\frac{(2+k^2)^2(12-4k^2+3k^4)^2}{(-2+k^2)^2}, \\
Y&= \frac{8(2+k^2)(-16+20k^2+k^6)(-4-5k^4+k^6)(12-4k^2+3k^4)}{(-2+k^2)^3}.
\end{align*}
We do not present explicitly the map $\phi:\;\cal{C}\rightarrow \cal{E}$ because
the formula is unwieldy. Note that the existence of $Q$ of infinite order on $\cal{E}$ implies
the Zariski density of rational points on the surface $\cal{E}$ (using the same reasoning as
in the proof of the previous theorem).
Computing $\phi^{-1}(mQ)$ for $m\in\Z$, and then the expressions for $x, z$, we get rational
parametric solutions of the system (\ref{prop3.4}). This observation finishes the proof.
\end{proof}

\noindent
\begin{rem}{\rm
The simplest parametric solution  of (\ref{prop3.4}) that we find is
\[ x=y=\frac{(4+2k-2k^2+k^3)(2-2k+k^2+k^3)(4-16k-12k^2-8k^3+k^4)}{2(2+k^2)^3(4-12k^2+k^4)}, \]
\[ z=\frac{(2-k^2)(16-352k^2-104k^4-88k^6+k^8)}{4(2+k^2)^3(4-12k^2+k^4))}. \]
}
\end{rem}

\section{Points in $\Q^3$ at rational distance to the vertices of the unit square}\label{sec4}

We consider here the problem of finding points in $\Q^3$ that lie at rational distance to the vertices
of the unit square, i.e. we do not assume any additional constraints on the coordinates of the points.
From the previous section we know that there is an infinite set $\cal{M}$ of rational curves lying in
the plane $x=1/2$ (or in the plane $x=y$) with the property that each rational point on each curve
$C\in\cal{M}$ has rational distance to the vertices of the unit square. A question arises whether in
the more general situation considered here we can expect the existence of rational surfaces having the
same property. Moreover, can any density result be obtained in this case? Unfortunately, we are unable
to prove any density result. However, we show that there are many rational points in $\Q^{3}$ lying at
rational distance to the vertices of the unit square. More precisely, we show the following.

\begin{thm}\label{unirational}
Put $A=\{(0,0,0),\;(0,1,0),\;(1,0,0),\;(1,1,0)\}$ and consider the set
\begin{equation*}
\cal{F}:=\{P\in \Q^3:\;\mbox{the distance}\;|PQ|\;\mbox{is rational for all}\;Q\in A\}.
\end{equation*}
Then the algebraic variety parameterizing the set $\cal{F}$ is unirational over $\Q$.
\end{thm}
%
\begin{proof}
It is clear that points in $\cal{F}$ are in one to one correspondence
with rational points on the intersection in $\R^{7}$ of the following four quadratic threefolds:
\begin{equation}\label{sys2}
\begin{cases}
\begin{array}{lll}
  x^2+y^2+z^2 & = & P^2,  \\
  (1-x)^2+y^2+z^2 & = & Q^2, \\
  x^2+(1-y)^2+z^2 & = & R^2, \\
  (1-x)^2+(1-y)^2+z^2 & = & S^2.
\end{array}
\end{cases}
\end{equation}
We immediately have
\begin{equation}\label{xysol}
x=\frac{1}{2}(P^2-Q^2+1),\quad y=\frac{1}{2}(P^2-R^2+1),
\end{equation}
and $P^2-R^2=Q^2-S^2$. All rational solutions of the latter are given by
\begin{equation*}
P=uX+Y,\quad Q=uX-Y,\quad R=uY+X,\quad S=uY-X,
\end{equation*}
and then from (\ref{xysol}),
\begin{equation*}
x=2 u X Y+\frac{1}{2},\quad y=\frac{1}{2}(u^2-1)(X^2-Y^2)+\frac{1}{2}.
\end{equation*}
Finding points on the system (\ref{sys2}) now reduces to studying the algebraic variety
$$
\cal{S}:\;V^2=G(u,X,Y),
$$
 where $V=2z$ and  the polynomial $G$ is given by
\begin{equation*}
G(u,X,Y)=-2+2(u^2+1)(X^2+Y^2)-(u^2-1)^2(X^2-Y^2)^2-16u^2X^2Y^2.
\end{equation*}
The dimension of $\cal{S}$ is 3. However, we can view the variety $\cal{S}$ as a del Pezzo surface of
degree two defined over the field $\Q(u)$. It is known then that the existence of a sufficiently general
$\Q(u)$-rational point on $\cal{S}$ implies $\Q(u)$-unirationality, and in consequence $\Q$-unirationality,
of $\cal{S}$ (see Manin~\cite[Theorem 29.4]{Man}). However, it seems that there is no general
$\Q(u)$-rational point on $\cal{S}$. Thus it is natural to ask how one can construct a rational base
change $u=\phi(t)$ such that the surface $\cal{S}_{\phi}:\;V^2=G(\phi(t),X,Y)$ defined over the field
$\Q(t)$, contains a $\Q(t)$-rational point. We present the following approach to this problem. Suppose that
$Q_{0}=(u_{0},X_{0},Y_{0},V_{0})$ is a rational point with non-zero coordinates lying on $\cal{S}$.
We construct a parametric curve $\cal{L}$ lying on $\cal{S}$ as follows. Define $\cal{L}$ by equations
\begin{equation*}
\cal{L}:\;u=u_{0},\;X=T+X_{0},\;Y=pT+Y_{0},\;V=qT^2+tT+V_{0},
\end{equation*}
where $t$ is a rational parameter and $p,q,T$ are to be determined. With $u,X,Y,V$ so defined,
$V^2-G(u,X,Y)=\sum_{i=1}^{4}A_{i}(p,q)T^i$. The expression $A_{1}$ is linear in $p$ and takes the form
$A_{1}=pB_{1}+B_{0}+2tV_{0}$, where $B_{0}, B_{1}$ depend only on the coordinates of the point $Q$.
In particular, $A_{1}$ is independent of $q$; so the equation $A_{1}=0$ has a non-zero
solution for $p$ if and only if $B_{1}\neq 0$. The expression for $B_1$ is
\begin{equation*}
B_{1}=4Y_{0}((-u_0^4+10u_0^2-1)X_0^2+(u_0^2-1)^2Y_0^2-u_0^2-1).
\end{equation*}
Next, observe that $A_{2}=C_{2}p^2+C_{1}p+C_{0}+2qV_{0}+t^2$, where $C_{i}$ depend only on the coordinates
of the point $Q_{0}$ for $i=0,1,2$, and thus $A_2=0$ can be solved for $q$ precisely when $V_0$ is non-zero.
To sum up, the system $A_{1}=A_{2}=0$ has a non-trivial solution for $p,q$ as rational functions in
$\Q(t)$ when $B_1V_0 \neq 0$.
With $p,q$ computed in this way:
 \begin{equation*}
V^2-G(u,X,Y)=T^3(A_{3}(p,q)+A_{4}(p,q)T).
\end{equation*}
If $A_{3}A_{4}\neq 0$ as a function in $t$ then the expression for $T$ that we seek is given by
$T=-A_{3}(p,q)/A_{4}(p,q)$. Thus if the point $Q_{0}=(u_0,X_0,Y_0,V_0)$ satisfies certain conditions,
then there exists a rational curve on the surface $\cal{S}_{u_{0}}:\;V^2=G(u_{0},X,Y)$. Moreover,
the curve constructed in this manner can be used to produce rational expressions for $P,Q,R,S$ and
in consequence rational expressions for $x,y,z$ satisfying the system (\ref{sys2}).

Let $X=X'(t), Y=Y'(t)$ be parametric equations of the constructed curve. The polynomial $G$
is invariant under the mapping $(u,X,Y)\mapsto \Big(\frac{X}{Y},uY,Y\Big)$ and thus we can define
a non-constant base change $u=\phi(t)=X'(t)/Y'(t)$ such that the surface
$\cal{S}_{\phi}:\;V^2=G(\phi(t),X,Y)$ contains the $\Q(t)$-rational point $(X,Y)=(u_{0}Y'(t),Y'(t))$.
Using the cited result of Manin we get $\Q(t)$-unirationality of $\cal{S}_{\phi}$ and in consequence
$\Q$-unirationality of $\cal{S}$.

Thus in order to finish the proof it suffices to find a suitable point $Q_0$ on the threefold $\cal{S}$.
It is straightforward to check that all the required conditions on $Q_0$ are met on taking
$$
(u_{0},X_{0},Y_{0},V_{0})=\Big(2,\frac{1}{12},\frac{19}{36},\frac{7}{27}\Big).
$$
With this choice of $Q_0$, the expressions for $x,y,z$ arising from the constructed parametric curve
are as follows:
\begin{align*}
x = & 3(5522066829177276301427600 - 258403606687492419505600t \\
& + 24350105869790104153088t^2 - 930272613423360964576t^3 \\
& + 39295267680627366536t^4 - 1085485845235095088t^5 \\
& + 24133448660417792t^6 - 401146604231320t^7 + 3899504263625t^8)/\Delta^2, \\
y = & 30(3992136439221148602640 - 6939554120499388567712t \\
& + 117488065643083258096t^2 - 13393876262858078048t^3 \\
& + 411476041942299568t^4 - 13249457441223848t^5 \\
& + 681815047971100t^6 - 7562115944888t^7 + 337499289355t^8)/\Delta^2, \\
z = & 714(3779374597422498556400 + 529318935972209201600t \\
& - 977278343015269168t^2 + 1745565618326470736t^3 \\
& - 10290117484952896t^4 + 1635035001144368t^5 \\
& - 3620551914412t^6 + 458263598420t^7 + 118863425t^8)/\Delta^2,
\end{align*}
where
$$
\Delta=18(221769748580 - 3052768504t + 670128264t^2 - 6059132t^3 + 500425t^4).
$$
\end{proof}
\begin{rem}
{\rm
The point $(x,y,z)$ satisfies $0<x,y<1$ for values of $t$ satisfying $t<-10.9337$, or $t>28.2852$.
}
\end{rem}
Notwithstanding the large coefficient size in the above parameterization,
there seem to be many points in $\Q^3$ at rational distance to the vertices $A$ of the unit square.
A (non-exhaustive) search finds the following points $(x,y,z) \in \Q^3$ of height at most $10^4$,
$x \neq \frac{1}{2}$, $x \neq y$, having rational distances to the vertices $(0,0,0)$, $(0,1,0)$,
$(1,0,0)$, $(1,1,0)$ of the unit square, and which lie in the positive octant. We list only one point
under the symmetries $x \leftrightarrow 1-x, \;\; y \leftrightarrow 1-y, \;\; x \leftrightarrow y$.
\begin{center}
\begin{tabular}{rrr|rrr}
$x$ & $y$ & $z$  &  $x$ & $y$ & $z$ \\ \hline
41/27 & 77/108 & 28/27 & 1/35 & 37/105 & 17/140 \\
5/54 & 35/108 & 7/54 & 161/80 & 587/300 & 7/25 \\
83/125 & 549/500 & 14/75 & 37/156 & 987/2704 & 119/676 \\
1/189 & 283/756 & 31/189 & 232/189 & 493/756 & 59/189 \\
113/190 & 2369/1900 & 287/2850 & 202/195 & 213/325 & 161/1300 \\
383/348 & 5397/1682 & 2429/1682 & 571/476 & 2419/2975 & 94/425 \\
203/594 & 119/1188 & 469/594 & 1589/594 & 985/1188 & 427/594 \\
1/756 & 127/1512 & 307/756 & 1436/847 & 7967/3388 & 992/847 \\
127/1029 & 341/1372 & 307/343 & 251/1029 & 401/1372 & 223/343 \\
791/1210 & 5299/3630 & 2569/2420 & 1571/1210 & 7487/7260 &  509/1210 \\
1906/2541 & 4019/3388 & 360/847 & 2185/2541 & 3819/3388 & 345/847 \\
3059/2738 & 4487/5476 & 3059/8214 & & & \\ \hline
\end{tabular}
\begin{center}
Table 2: points in $\Q^3$ at rational distance to the vertices of the unit square
\end{center}
\end{center}

\bigskip

We are motivated to make the following conjecture.
\begin{conj}
Put $A=\{(0,0,0),\;(0,1,0),\;(1,0,0),\;(1,1,0)\}$ and consider the set
\begin{equation*}
\cal{F}:=\{P\in \Q^3:\;\mbox{the distance}\;|PQ|\;\mbox{is rational for all}\;Q\in A\}.
\end{equation*}
Then $\cal{F}$ is dense in $\R^3$ in the Euclidean topology.
\end{conj}

\section{Points with rational distances from the vertices of a tetrahedron}\label{sec5}

Let $P_{0}, P_{1}, P_{2}, P_{3}$ be given points in $\Q^{3}$, not all lying on a plane.
Without loss of generality we may assume that $P_{0}=(0,0,0)$. Put
$P_{i}=(a_{i1},a_{i2},a_{i3})$ for $i=1,2,3$, and define $d_{ij}=|P_{i}P_{j}|$ for $0 \leq i<j \leq 3$,
i.e. $d_{ij}$ is the distance between the points $P_{i}, P_{j}$.
The constraint on the points $P_i$ implies that the points define the vertices
of a genuine tetrahedron with non-zero volume, so that the determinant of the matrix
$[a_{ij}]_{1\leq i,j\leq 3}$ is non-zero.
Let $P=(x,y,z)$ be a point in $\Q^{3}$
with rational distance to each of the points $P_{0}$, $P_{1}$, $P_{2}$, $P_{3}$. The corresponding system
of Diophantine equations is thus
\begin{equation*}
\begin{cases}
\begin{array}{lll}
  x^2+y^2+z^2 & = & Q_{0}^2 \\
  (x-a_{i1})^2+(y-a_{i2})^2+(z-a_{i3})^2 & = & Q_{i}^2,\quad \mbox{for}\; i=1,2,3,
\end{array}
\end{cases}
\end{equation*}
or equivalently, on replacing the second, third, and fourth equations by their differences with the first
equation,
\begin{equation}\label{generalsys}
\begin{cases}
\begin{array}{lll}
  x^2+y^2+z^2 & = & Q_{0}^2 \\
  a_{i1}x+a_{i2}y+a_{i3}z & = & \frac{1}{2}(Q_{0}^2-Q_{i}^2+d_{0i}^{2}),\quad \mbox{for}\; i=1,2,3.
\end{array}
\end{cases}
\end{equation}
Since the determinant of the matrix $A=[a_{ij}]_{1\leq i,j\leq 3}$ is non-zero, the (linear) system
consisting of the last three
equations from (\ref{generalsys}) can be solved with respect to $x,y,z$. The solution takes
the following form:
\begin{equation}
\label{xyz}
x=\frac{\op{det}A_{1}}{\op{det}A},\quad y=\frac{\op{det}A_{2}}{\op{det}A},\quad z=\frac{\op{det}A_{3}}{\op{det}A}
\end{equation}
where $A_{i}$, for $i=1,2,3$, is obtained from the matrix $A$ by replacing the $i$-th column by the column
comprising the right hand sides of the last three equations from (\ref{generalsys}). In particular,
$x,y,z$ are (inhomogeneous) quadratic forms in four variables $Q_{0}$, $Q_{1}$, $Q_{2}$, $Q_{3}$, with
coefficients in $\mathbb{K}:=\Q(\{a_{ij}:\;i,j\in\{1,2,3\}\})$.  Putting these computed values of
$x$,$y$,$z$ into the first equation, there results one inhomogeneous equation of degree four in four
variables. We homogenize this equation by introducing new variables $Q_{i}=R_{i}/R_{4}$ for $i=0,1,2,3$,
and work with the quartic threefold, say $\cal{X}$, defined by an equation of the form $\cal{F}({\bf R})=0$,
where for ease of notation we put ${\bf R}=(R_{0},R_{1},R_{2},R_{3},R_{4})$. Using Mathematica, the
set $\op{Sing}(\cal{X})$ of singular points of the variety $\cal{X}$ is computed to be
\begin{align*}
\op{Sing}(\cal{X})=\{&(0,\pm d_{01},\pm d_{02}, \pm d_{03},1), \; (\pm d_{01}, 0, \pm d_{12}, \pm d_{13},1),  \\
                     &(\pm d_{02},\pm d_{12}, 0, \pm d_{23},1), \; (\pm d_{03},\pm d_{12}, \pm d_{23}, 0, 1), \; (1,\pm 1, \pm 1, \pm 1, 0)\}.
\end{align*}
Thus for generic choice of $P_{1}$, $P_{2}$, $P_{3}$, the variety $\cal{X}$ contains 40 isolated singular
points.

We now prove that for generic choice of $P_{1}$, $P_{2}$, $P_{3}$, there is a solution depending on three
(homogenous) parameters of the equation defining the variety $\cal{X}$. We thus regard $a_{ij}$ as
independent variables and work with $\cal{X}$ as a quartic threefold defined over the rational function
field $\mathbb{K}$. In order to find a parameterization we will use the rational double point
$P=(1,1,1,1,0)$ lying on $\cal{X}$ and the idea used in the proof of Theorem \ref{unirational}. Put
\begin{equation*}
R_{0}=T+1,\quad R_{i}=(p_{i}+1)T+1,\quad\mbox{for}\;i=1,2,3,\mbox{ and}\quad R_{4}=p_{4}T,
\end{equation*}
where $p_{i}$ and $T$ are to be determined. On substituting these expressions into the equation $\cal{F}({\bf R})=0$,
there results $T^2(C_{2}+C_{3}T+C_{4}T^2)=0$, where $C_{i}$ is a homogenous form of degree $i$ in the four
variables $p_{1},\ldots,p_{4}$. Certainly under the assumption on the points $P_{i}, i=0,1,2,3$ (namely,
$\op{det}A\neq0$), the form $C_{2}$ is non-zero as an element of $\mathbb{K}[p_{1},p_{2},p_{3},p_{4}]$. Indeed, we have $C_{2}(0,0,0,p_{4})=-(\op{det}A)^2p_{4}^2$.
We also checked that for a generic choice of the points $P_{1}$, $P_{2}$, $P_{3}$, the polynomial $C_{2}$
is genuinely dependent upon the variables $p_{1},\ldots,p_{4}$, in that there are no linear forms
$L_{j}(p_1,...,p_4)$, $j=1,2,3$, such that $C_{2}(L_1,L_2,L_3)$ is a form in three or fewer variables.

Consider now the quadric $\cal{Y}:\;C_{2}(p_1,p_2,p_3,p_{4})=0$, regarded
as a quadric defined over $\mathbb{K}$. There are $\mathbb{K}$-rational points on $\cal{Y}$, namely
$Y_{j}=(a_{1j},a_{2j},a_{3j},1)$, $j=1,2,3$, and so in particular, $\cal{Y}$ can be rationally
parameterized with parametrization of the form $p_{i}=X_{i}(q_1,q_2,q_3)$, for homogeneous quadratic
forms $X_{i}$, $i=1,2,3,4$.  Thus, after the substitution $p_{i}\rightarrow X_{i}$ there results
an equation $T^3(C'_{3}+C'_{4}T)=0$, where $C'_{i}=C_{i}(X_1,X_2,X_3,X_4)$ and $C'_{i}\neq 0$ as an element
of $\mathbb{K}[q_{1}, q_{2}, q_{3}]$ for $i=3,4$.  This equation has a non-zero $\mathbb{K}$-rational root
$T=\phi(q_1,q_2,q_3)=-C'_{3}/C'_{4}$ and accordingly we get a rational parametric solution in three
(homogenous) parameters of the equation defining $\cal{X}$, in the form
\begin{equation*}
Q_{0}=\frac{1}{X_{4}({\bf q})}\Big(1+\frac{1}{\phi({\bf q})}\Big),\quad Q_{i}=\frac{1}{X_{4}({\bf q})}\Big(1+X_{i}({\bf q})+\frac{1}{\phi({\bf q})}\Big), \quad i=1,2,3,
\end{equation*}
where we put ${\bf q}=(q_1,q_2,q_3)$.
It is straightforward to check that the image of the map
$\Phi:\;\mathbb{P}(\mathbb{K})^2\ni (q_{1},q_{2},q_{3})\mapsto (Q_{0},Q_{1},Q_{2},Q_{3})\in\cal{X}(\mathbb{K})$
is not contained in a curve lying on the variety $\cal{X}$. Using now the expressions for $Q_{0}$,
$Q_{1}$, $Q_{2}$, $Q_{3}$, we can recover the corresponding expressions for $x,y,z$ given by (\ref{xyz}).

It is possible to write down from the Jacobian matrix of ${\bf R}(q_1,q_2,q_3)$ all the conditions
on $\{a_{ij}\}$,  $i,j\in\{1,2,3\}$, which guarantee that
the parameterization is genuinely dependent on three (homogenous) parameters. However, we refrain from
doing so, because the computation is massively memory intensive, and the resulting equations
complicated and unenlightening.  If we choose particular values of $a_{ij}$, then this independence
of $q_1,q_2,q_3$ is readily checked
(as happens, for example, when $P_1=(1,0,0)$, $P_2=(0,1,0)$, $P_3=(0,0,1)$).
In general, there results an explicit rational parameterization in three independent
parameters. There may, however, be some choices of vertices $P_i$ for which this approach (with the particular
rational double point chosen in the construction) results in the image of the map $\Phi$ being
a curve lying on $\cal{X}$. \\

To sum up, we have the following result.

\begin{thm}\label{unirationlity2}
Let $P_{0}=(0,0,0)$ and let $P_{i}=(a_{i1},a_{i2},a_{i3})$ be generic points in $\Q^3$ for $i=1,2,3$. Then the
variety parameterizing the points $P\in\Q^3$ with rational distances to $P_{i}$, $i=0,1,2,3$ is a quartic
threefold $\cal{X}$; and the set of rational parametric solutions of the equation defining $\cal{X}$ is
non-empty.
\end{thm}

We believe that much more is true.

\begin{conj}\label{uniconj}
Let $P_{0}=(0,0,0)$ and $P_{1}$, $P_{2}$, $P_{3}$ be generic rational points such that no three lie on a line
and the points do not all lie on a plane. Then the variety, say $\cal{X}$, parameterizing those
$P\in\Q^3$ with rational distances to $P_{i}$, $i=0,1,2,3$, is unirational over $\Q$.
\end{conj}

One can also state the following natural question.

\begin{ques}
Let $\cal{X}$ be defined as in Conjecture \ref{uniconj}. Is the set $\cal{X}(\Q)$ dense in the
Euclidean topology in the set $\cal{X}(\R)$?
\end{ques}

We expect that the answer is yes.

\begin{rem}
{\rm The construction above finds a double infinity of points in $\Q^3$ at rational distance from the
four vertices of the tetrahedron. If we suppose that the initial tetrahedron has rational edges, then
we thus deduce infinitely many sets of five points in $\Q^3$ where the ten mutual distances are
rational. We take as example the tetrahedron with vertices
\[ P_{1}=(0,0,0),\quad P_{2}=(1,0,0),\quad P_{3}=\left(\frac{11}{200},\frac{117}{800},0\right),\quad P_{4}=\left(\frac{7}{25}, \frac{63}{325}, \frac{21}{260}\right).\]
This is chosen as an example of a tetrahedron, discovered by Rathbun, where the edges, face areas,
and volume, are all rational. It corresponds to the first example in the list in Section D22 of
Guy~\cite{Guy}. The explicit parametrization as computed above takes several computer screens to
display, so we do not present it. However, on computing specializations, the point with smallest
coordinates (minimizing the least common multiple of the denominators of $x,y,z$) that we could find is
\[ \left( \frac{617}{4900}, \; \frac{2553}{63700}, \; \frac{3}{25480} \right), \]
which in fact lies within the tetrahedron.
}
\end{rem}

\begin{rem} {\rm
The fact in the above proof that the matrix $A=[a_{ij}]_{1\leq i,j\leq 3}$ is non-singular
follows from the assumption that the points $P_0,P_1,P_2,P_3$ define a genuine tetrahedron.
A question arises as to what can be said in the situation when $\op{det}A=0$?
We need to consider two cases: where the rank $\op{rk}(A)$ is 2 or 1,
corresponding respectively
to the four points being coplanar, and the four points being collinear.
Consider first the case of $\op{rk}(A)=2$. Note that we encounter this situation in section \ref{sec4}.
The vectors $P_{1}, P_{2}, P_{3}$ are linearly dependent, and without loss
of generality we can assume that $P_{1}, P_{2}$ are linearly independent, so that
$P_{3}=pP_{1}+QP_{2}$ for some $p, q\in\Q$. It follows that the linear forms in $x, y, z$ from
the system (\ref{generalsys}) are linearly dependent.
Let $A_{ij}$ be the $2\times 2$ matrix obtained from $A$ by deleting the $i$-th row and the $j$-th column.
Then at least one of $A_{31}$, $A_{32}$, $A_{33}$ has non-zero determinant. Without loss of generality,
suppose $\op{det}A_{31} \neq 0$.
Solving the first two equations at (\ref{generalsys}) with respect to $y, z$:
\begin{align*}
y=&-\frac{\op{det}A_{32}}{\op{det}A_{31}}x-\frac{1}{2\op{det}A_{31}}(a_{23} d_{01}^2-a_{13} d_{02}^2+\left(a_{23}-a_{13}\right) Q_0^2-a_{23} Q_1^2+a_{13} Q_2^2),\\
z=&-\frac{\op{det}A_{33}}{\op{det}A_{31}}x-\frac{1}{2\op{det}A_{31}}(a_{22} d_{01}^2-a_{12} d_{02}^2+\left(a_{22}-a_{12}\right) Q_0^2-a_{22} Q_1^2+a_{12} Q_2^2.
\end{align*}
Moreover, $Q_{0}, Q_{1}, Q_{2}, Q_{3}$ need to satisfy the equation
\begin{equation}\label{quadric1}
\cal{Q}:\;(p+q-1)Q_{0}^2-pQ_{1}^2-qQ_{2}^2+Q_{3}^2=d_{03}^2-pd_{01}^2-qd_{02}^2.
\end{equation}
The quadric $\cal{Q}$ may be viewed as a quadric defined over the function field $\mathbb{K}:=\Q(\{a_{ij}:\;i=1,2,j=1,2,3\})$.
The quadric $\cal{Q}$ contains the point at infinity $(Q_{0}:Q_{1}:Q_{2}:Q_{3}:T)=(1:1:1:1:0)$
and thus $\cal{Q}$ can be parameterized by rational functions, say $Q_{i}=f_{i}({\bf R}) \in \mathbb{K}({\bf R})$,
where ${\bf R}=(R_{0},R_{1},R_{2})$ are (non-homogenous) coordinates.


Moreover, the numerator
of $f_{i}$ is of degree $\leq 2$ for $i=0,1,2$; and the same is true for the common denominator
of $f_{i}, i=0,1,2$. Using this parametrization we compute the expressions for $y, z$. Next,
substitute the computed values of $y, z$ and $Q_{0}$ into the equation $x^2+y^2+z^2=Q_{0}^2$.
This equation is a quadratic equation in $x$ of the form
\begin{equation*}
C_{2}x^2+C_{1}x+C_{0}=0,
\end{equation*}
where $C_{i}\in\mathbb{K}({\bf R})$ for $i=0,1,2$. We arrive at the problem of finding rational points
on the threefold
\begin{equation*}
\cal{X}:\;V^2=C_{1}^2-4C_{0}C_{2}=:F({\bf R})
\end{equation*}
defined over the field $\mathbb{K}$. The polynomial $F$ is of degree 6. However, one can check that
with respect to each $R_{i}, i=0,1,2$, the degree of $F$ is 4, and thus $\cal{X}$ can be viewed as
a hyperelliptic quartic (of genus $\leq 1$) defined over the field $\mathbb{K}({\bf R}')$, where ${\bf R}'$ is
a vector comprising exactly two variables from $R_{0},R_{1}, R_{2}$. We thus expect that
for most rational points $P_{1}, P_{2}, P_{3}$ with $P_{3}=pP_{1}+qP_{2}$, there is a specialization
of $R_{0}, R_{1}$ (say), to rational numbers such that $\cal{X}_{R_{0}, R_{1}}$
represents a curve of genus one with infinitely many rational points. Tracing back the reasoning
in this case we will get infinitely many rational points with rational distances to the points
$P_{0}, P_{1}, P_{2}, P_{3}$.

\bigskip

What can be done in the case when $\op{rk}(A)=1$ (which corresponds to the points $P_0,P_1,P_2,P_3$
being collinear)? In order to simplify the notation, put $P_{1}=(a,b,c)$ and $d_{01}=d$. Without loss
of generality we can assume that $P_{2}=pP_{1}, P_{3}=qP_{1}$ for some $p, q\in\Q\setminus\{0\}$.
Then the system  (\ref{generalsys}) comprises just one linear form in $x, y, z$ which needs to be
represented by three non-homogenous quadratic forms. More precisely,
\begin{equation}\label{sys3}
ax+by+cz=\frac{1}{2}(Q_{0}^2-Q_{1}^2+d^2)=\frac{1}{2p}(Q_{0}^2-Q_{2}^2+p^2d^2)=\frac{1}{2q}(Q_{0}^2-Q_{3}^2+q^2d^2).
\end{equation}
Let $\cal{V}$ be the variety defined by the last two equations. After homogenization by
$Q_{i}\mapsto Q_{i}/T$ and simple manipulation, we get
\begin{equation*}
 \cal{V}:\; \
\begin{cases}
\begin{array}{lll}
Q_{2}^2=(1-p)Q_{0}^2+pQ_{1}^2+p(p-1)d^2T^2, \\
Q_{3}^2=(1-q)Q_{0}^2+qQ_{1}^2+q(q-1)d^2T^2.
\end{array}
\end{cases}
\end{equation*}
The point $(Q_{0}:Q_{1}:Q_{2}:Q_{3}:T)=(1:1:1:1:0)$ lies on $\cal{V}$ and can be used to find
parametric solutions of the system defining $\cal{V}$. However, observe that any point which lies
on $\cal{V}$ with $T\neq 0$ allows us to compute the value of $z$ from equation (\ref{sys3}).
This expression for $z$ depends on $x, y$, and substituting into the first equation
at (\ref{generalsys}), namely $x^2+y^2+z^2=Q_{0}^2$, we are left with one equation of the form
$$
\cal{W}:\;C_{0}x^2+C_{1}xy+C_{2}y^2+C_{3}x+C_{4}y+C_{5}=0,
$$
where $C_{i}$ depends on $p,q,a,b,c$ and the solution of the system defining the variety $\cal{V}$.
In general, $\cal{W}$ is a conic and thus has genus 0. Thus, provided that we can find a rational
point on $\cal{W}$, we can find infinitely many rational points (in fact a parameterized curve)
with rational distances to the four collinear points $P_{0}, P_{1}, P_{2}, P_{3}$.
As example here, assume that $d=\sqrt{a^2+b^2+c^2}$ is a rational number.
Then the variety $\cal{V}$ contains the rational line
\begin{equation*}
(Q_{0}:Q_{1}:Q_{2}:Q_{3}:T)=(u-d/2:u+d/2:u - (1/2 - p)d:u - (1/2 - q)d:1).
\end{equation*}
In this case (\ref{sys3}) reduces to the one equation $ax+by+cz=d(d-2u)/2$. Solving for
$z$, and performing the necessary computations, the equation for the quadric $\cal{W}$ takes
the following form:
\begin{equation*}
V^{2}=b^2 + c^2 - d^2+4adX-4(a^2 + b^2 + c^2)X^2=-a^2+4adX-4d^2X^2=-(a-2dX)^2,
\end{equation*}
where $V=(2(b^2+c^2)y-b(d^2-2du-2ax))/(c(d-2u))$ and $X=x/(d-2u)$, the last identity following
from the equality $a^2+b^2+c^2=d^2$. From the assumption on rationality of $d$, we can find
$x, y, z$ in the following form:
\begin{equation*}
x=\frac{a(d-2u)}{2d},\quad y=\frac{b(d-2u)}{2d},\quad z=\frac{c(d-2u)}{2d},
\end{equation*}
with
\begin{equation*}
Q_0=\frac{d-2u}{2}, \quad Q_1=\frac{d+2u}{2}, \quad Q_2=d p-\frac{d-2u}{2}, \quad Q_3=d q-\frac{d-2u}{2},
\end{equation*}
giving rational solutions of the original system.

\bigskip


}
\end{rem}

\begin{rem}
\rm{
Guy op. cit. gives one parameterized family of tetrahedra which have rational edges, face areas, and volume.
He also lists nine examples due to John Leech of such tetrahedra comprised of four congruent
acute-angled Heron triangles appropriately fitted together. The six edges of the tetrahedron
thus fall into three equal pairs. We discover that it is straightforward to write down an
infinite family of such tetrahedra as follows.

If the Heron triangle has sides $p,q,r$, then the area and volume conditions for the tetrahedron become
\begin{align*}
(p+q+r)(-p+q+r)(p-q+r)(p+q-r)= & \square, \\
2(-p^2+q^2+r^2)(p^2-q^2+r^2)(p^2+q^2-r^2)= & \square.
\end{align*}
Using the Brahmagupta parameterization of Heron triangles, we set
\[ (p,q,r)=((v+w)(u^2-v w), \; v(u^2+w^2), \; w(u^2+v^2) ), \]
reducing the two conditions above to the single demand
\[ -(u^2-v^2)(u^2-w^2)(u^2-u(v+w)-v w)(u^2+u(v+w)-v w) = \square. \]
Setting $W=w/u$, this is equivalent to
\[ -(1-W^2) \left(  \frac{u+v}{u-v} -W \right) \left( \frac{u-v}{u+v} + W \right) = \square. \]
This elliptic quartic has cubic form
\[  Y^2 = X(X + v^2(u^2-v^2))(X - u^2(u^2-v^2)). \]
Demanding a point with $X=2u v^2(u+v)$ gives
\[ 2(3u-v)(-u+2v) = \square, \]
parameterized by
\[ (u,v,w)=( m^2+4, 3 m^2+2), \mbox{ with } w=\frac{(2m^2+3)(m^2+4)}{4m^2+1}. \]
This in turn leads to the tetrahedron with vertices
\begin{align*}
P_1= & (0, \; 0,\; 0), \\
P_2= & (10(m^4-1)(m^4+3m^2+1), \; 0, \; 0), \\
P_3= & \big(\frac{2(m^2-1)(m^2+4)(3m^2+2)^2}{5}, \frac{(m^2+4)(2m^2+3)(3m^2+2)(4m^2+1)}{5}, \; 0\big), \\
P_4= & \big( \frac{2(m^2-1)(2m^2+3)^2(4m^2+1)}{5}, \\
& -\frac{(2m^2+3)(2m^2-5m-2)(2m^2+5m-2)(3m^2+2)}{5}, \\
& \; 4(m^2-1)m(2m^2+3)(3m^2+2) \big);
\end{align*}
edge lengths $(p,q,r)$ given by
\begin{align*}
p & = 10(m^4-1)(m^4+3m^2+1), \\
q & = (m^2+4)(3m^2+2)(2m^4+2m^2+1), \\
r & = (2m^2+3)(4m^2+1)(m^4+2m^2+2);
\end{align*}
face areas given by
\[ (m^4-1)(m^2+4)(4m^2+1)(2m^2+3)(3m^2+2)(1+3m^2+m^4); \]
and volume equal to
\[\frac{1}{62208} m(m^2-1)(m^4-1)(m^2+4)(4m^2+1)(2m^2+3)^2(3m^2+2)^2(1+3m^2+m^4). \]
}
\end{rem}

\section{The unit cube}\label{sec6}

Finding an infinity of points in $\Q^3$, if indeed such exist, that lie at rational distance from the {\it eight} vertices $(i,j,k)$, $i,j,k=0,1$,
of the unit cube seems to be an intractable problem. If we restrict attention to the plane $x=y$,
we are aware of the following two points (equivalent under the symmetry $x \leftrightarrow 1-x$) where distances to the vertices
of the unit square are rational, and distances to the two cube vertices $(1,0,1), (0,1,1)$ are rational:
\begin{equation}
\label{pts6}
(x,y,z)=\left( \frac{31}{108}, \frac{31}{108}, \frac{1519}{1080} \right), \qquad \left( \frac{77}{108}, \frac{77}{108}, \frac{1519}{1080} \right).
\end{equation}
The defining system of equations for this situation is
\begin{equation*}
\begin{cases}
 \begin{array}{lll}
2 x^2 + z^2         & = & P^2, \\
2 x^2-2x+1+z^2      & = & Q^2, \\
2 (x-1)^2 +z^2      & = & S^2, \\
(1-x)^2+x^2+(1-z)^2 & = & T^2.
\end{array}
\end{cases}
\end{equation*}
Then $1+Q^2-2z=T^2$, so we obtain
\[  z^2 = 1/2 (1-8 m^2/t^2)(-1+2(1-m^2)^2/t^2), \qquad 1+(1+m^2)^2/t^2 - 2 z = T^2. \]
Equivalently,
\[ 2(t^2-8 m^2)(-t^2+2(1-m^2)^2) = (t^2 + (1+m^2)^2 - U^2)^2, \]
where $U=Tt$, $Z=t^2 z$.
A search over this surface up to a height of $5000$ resulted in discovering only the point
$(m,t,U)=(-24,360,313)$ and symmetries, leading to the points at (\ref{pts6}). \\ \\
A point on the plane $x=1/2$ at rational distance from the vertices of the cube results in
four pairs of equal distance, with defining equations
\begin{equation}\label{symsys}
\begin{cases}
 \begin{array}{lll}
1/4+y^2+z^2         & = & P^2, \\
1/4+(1-y)^2+z^2     & = & Q^2, \\
1/4+y^2+(1-z)^2     & = & R^2, \\
1/4+(1-y)^2+(1-z)^2 & = & S^2,
\end{array}
\end{cases}
\end{equation}
and we found no solution.
However, if we ask only for {\it three} pairs of distances to the
cube vertices be rational, rather than four, e.g. consider the system of equations defined by the
first three equations from (\ref{symsys}), then we are able to prove the following result.
\begin{thm}
Let $\cal{A}$ be the set of rational curves lying on the plane $x=1/2$ with the property that each rational
point on $A\in\cal{A}$ has rational distances to six vertices of the unit cube. Then $\cal{A}$ is infinite.
\end{thm}
\begin{proof}
We consider only the system of equations defined by the first three equations from the system
(\ref{symsys}) above (other cases are treated in the same manner).
The six distances now fall into three equal pairs, requiring
\[ y=(P^2-Q^2+1)/2, \quad z=(P^2-R^2+1)/2, \]
together with the equation which can be written in homogenous coordinates in the following form:
\begin{equation*}
\label{PQRT}
\cal{V}:\;(P^2 - R^2)^2 + (P^2 - Q^2)^2 -2 (Q^2+R^2) T^2 +3 T^4 = 0.
\end{equation*}

We prove that the set of rational curves lying on $\cal{V}$ is infinite. Consider the intersection of
$\cal{V}$ with the family of planes $L_{a}:\;T=a(P-R)$. Remarkably, the intersection $\cal{V}\cap L_{a}$
defines a singular curve, say $\cal{C}$, in the projective plane $\mathbb{P}^{2}(\Q(a))$, with singular
points $[P:Q:R]=[1:\pm 1:1]$. In fact, the curve $\cal{C}$ is of genus 1.
By homogeneity we can assume that $R=1$. Making a change of variables
$$
(P,Q)=(p+1,\;pq+1)\quad\mbox{with inverse}\quad (p,q)=\Big(P-1,\;\frac{Q-1}{P-1}\Big)
$$
the (inhomogenous) equation of $\cal{C}$ takes the form $p^2H(p,q)=0$, where
\begin{equation*}
 H(p,q)=(2+3a^4-2(a^2+1)q^2+q^4)p^2+4(2-(1+a^2)q-q^2+q^3)p+4(q^2-2q-a^2+2)).
\end{equation*}
In other words, the curve $\cal{C}$ is the set-theoretic sum of the (double) line $p=0$ and the curve
of degree 6, given by the equation $\cal{C}':\;H(p,q)=0$. The equation for $\cal{C}'$ can be rewritten
as
\begin{equation*}
 \cal{C}':\;W^2=(a^2-1)q^4-2(a^2-1)q^3-(2a^2-1)^2q^2+2a^2(3a^2-2)q+a^2(3a^4-6a^2+2),
\end{equation*}
where we put $W=\frac{1}{2}(q^4-2(a^2+1)q^2+3a^4+2)p+q^3-q^2-(a^2+1)q+2$.
In order to guarantee the existence of rational points on $\cal{C}'$ (and hence on $\cal{C}$)
we consider a quadratic base change $a=(t^2+1)/2t$.  Then $a^2-1=((t^2-1)/2t)^2$ and thus the
curve $\cal{C}'$ contains a $\Q(t)$-rational point at infinity.
The birational model $\cal{E}'$ of the curve $\cal{C}'$ is given by the equation in short
Weierstrass form
$\cal{E}':\;Y^2=X^3+AX+B$,
where
\begin{align*}
A&=-108(13t^{16}-20t^{12}+78t^8-20t^4+13),\\
B&=864(23t^{24}-132t^{20}+129t^{16}-296t^{12}+129t^8-132t^4+23).
\end{align*}
The curve $\cal{E}'$ contains the point of infinite order
\begin{equation*}
Z=(12(2t^8+3t^6-2t^4+3t^2+2), \; 108t(t^2 + 1)(t^8 - 1)).
\end{equation*}
The point $2Z$ leads to a non-trivial curve lying on $\cal{V}$ (the equations for this curve
are too unwieldy to present explicitly here), and correspond to the following $y,z$ satisfying the first
three equations of our system:
\begin{align*}
y=&(t^{48}-8 t^{47}+20 t^{46}+8 t^{45}-24 t^{44}-1528 t^{43}+6684 t^{42}-4872 t^{41}-69302 t^{40}\\
  &+96040 t^{39}+771532 t^{38}-2467368 t^{37}-4047800 t^{36}+22047704 t^{35}+12635044 t^{34} \\
& -107433944 t^{33} -23948593 t^{32} +342788016 t^{31}+24622088 t^{30}-780080048 t^{29} \\
& -638000 t^{28}+1324015696 t^{27} -37969832 t^{26} -1716035152 t^{25}+57538508 t^{24} \\
& +1716035152 t^{23}-37969832 t^{22} -1324015696 t^{21} -638000 t^{20} +780080048 t^{19} \\
& +24622088 t^{18}-342788016 t^{17}-23948593 t^{16}+107433944 t^{15} +12635044 t^{14} \\
& -22047704 t^{13}-4047800 t^{12}+2467368 t^{11}+771532 t^{10}-96040 t^9-69302 t^8 \\
& +4872 t^7+6684 t^6 +1528 t^5-24 t^4-8 t^3+20 t^2+8 t+1)/(2 t \Delta), \\
z&=(3 t^{48}-16 t^{47}+56 t^{46}-32 t^{45}+1096 t^{44}-5696 t^{43}+ 15928 t^{42}+11472 t^{41} \\
& +51710 t^{40}-551056 t^{39}+1282392 t^{38}+3181248 t^{37}-11188440 t^{36}-701152 t^{35} \\
& +39387992 t^{34}-55013168 t^{33} -75669523 t^{32}+272885472 t^{31}+75471984 t^{30} \\
& -744371648 t^{29}+210064 t^{28}+1377115648 t^{27} -116092816 t^{26} -1850031968 t^{25} \\
& +173321252 t^{24}+1850031968 t^{23}-116092816 t^{22} -1377115648 t^{21} +210064 t^{20} \\
& +744371648 t^{19}+75471984 t^{18}-272885472 t^{17}-75669523 t^{16}+55013168 t^{15} \\
& +39387992 t^{14} +701152 t^{13} -11188440 t^{12}-3181248 t^{11}+1282392 t^{10} \\
& +551056 t^9 +51710 t^8 -11472 t^7 +15928 t^6 +5696 t^5+1096 t^4 +32 t^3 +56 t^2 \\
& +16 t+3)/((t^2-1)\Delta),
\end{align*}
where
\begin{align*}
\Delta & =4(t-1)(t+1)(t^2+1)^2(t^8-4 t^7+10 t^6+12 t^5-14 t^4-12 t^3+10 t^2+4 t+1) \; \times \\
& (t^{16}-4 t^{14}+168 t^{12} -492t^{10}+718 t^8-492 t^6+168 t^4-4 t^2+1) \times (t^{16}+4 t^{14} \\
& -32 t^{13} +232 t^{12}+160 t^{11}-756 t^{10}-320 t^9 +1102 t^8+320 t^7-756 t^6-160 t^5 \\
& +232 t^4+32 t^3+4 t^2+1).
\end{align*}
Computing the points $mZ$ for $m=3,4,\ldots$ and the corresponding points on $\cal{C}$, we get infinitely
many rational curves lying on $\cal{V}$; and the result follows.
\end{proof}

\bigskip


\bigskip

We know (up to symmetry) precisely two points $(x,y,z) \in \Q^3$, with $x \neq \frac{1}{2}$, $x \neq y$,
where the distances to the vertices $(0,0,0)$, $(0,1,0)$, $(1,0,0)$, $(1,1,0)$ of the unit square
are rational, and where there is also rational distance to a fifth vertex $(0,0,1)$ of the unit cube:

\bigskip

\begin{center}
\begin{tabular}{rrr|rrrrr}
$x$ & $y$ & $z$ & $d_1$ & $d_2$ & $d_3$ & $d_4$ & $d_5$ \\ \hline
77/108 & 41/27 & -28/27  &           71/36 & 67/36 & 49/36 & 43/36 & 95/36 \\
83/125 & -49/500 & -14/75 &    389/300 & 349/300 & 209/300 & 119/300 & 409/300 \\ \hline
\end{tabular}
\begin{center}
Table 3: points in $\Q^3$ with five rational distances to vertices of the unit cube
\end{center}
\end{center}

\bigskip

\noindent School of Mathematical and Statistical Sciences, Arizona State University, Tempe AZ 85287-1804,
USA; email:\; {\tt bremner@asu.edu}

\bigskip

\noindent Jagiellonian University, Faculty of Mathematics and Computer Science, Institute of Mathematics, {\L}ojasiewicza 6, 30 - 348 Krak\'{o}w, Poland;  email: {\tt maciej.ulas@uj.edu.pl}
\end{document}